\documentclass[10pt]{amsart}

\usepackage{amsfonts,amssymb,amsmath,amsthm,amscd,enumerate}
\usepackage{latexsym}
\usepackage{euscript}
\usepackage[english]{babel}
\usepackage[latin1]{inputenc}

\newtheorem{theorem}{Theorem}[section]
\newtheorem{definition}[theorem]{Definition}
\newtheorem{lemma}[theorem]{Lemma}
\newtheorem{corollary}[theorem]{Corollary}
\newtheorem{proposition}[theorem]{Proposition}
\newtheorem{Remark}[theorem]{Remark}

\title[Grassmann algebra over a finite field]{$\mathbb{Z}_{2}$-graded identities of the Grassmann algebra over a finite field}
\author{Luís Felipe Gonçalves Fonseca}
\address{Departamento de Matemática, Universidade Federal
de Viçosa - Campus Florestal, Rodovia LMG 818, km 06, Florestal, MG,
Brazil} \email{luisfelipe@ufv.br}

\begin{document}
\maketitle

\begin{abstract}

Let $F$ be a finite field with the characteristic $p > 2$ and let
$G$ be the unitary Grassmann algebra generated by an infinite
dimensional vector space $V$ over $F$. In this paper, we determine a
basis for $\mathbb{Z}_{2}$-graded polynomial identities for any
non-trivial $\mathbb{Z}_{2}$-grading such that its underlying vector
space is homogeneous.


\end{abstract}

\section{Introduction}


Grassmann algebra is an algebraic structure that arises in linear
algebra. The Grassmann algebra assumed its own importance even in
other areas of science such as physics, and geometry. It is also
important in PI-theory. A celebrated result obtained by Kemer,
depicted in 1987, \cite{Kemer}, shows that any associative
PI-algebra, over a field $F$ of characteristic zero, is
PI-equivalent to the Grassmann envelope of a finite- dimensional
associative super-algebra.

In the 1970s, Regev and Krakovsky \cite{Regev} described the
identities of the Grassmann algebra over a field of characteristic
zero. Almost two decades later, Giambruno and Koshlukov
\cite{Giambruno2} identified a basis for the identities of the
Grassmann algebra over an infinite field of characteristic $p
> 2$. Briefly, when the ground field is infinite and its
characteristic is not equal to two, the identities of the Grassmann
algebra follow from the triple commutator.

When the ground field is finite, its characteristic is $p > 2$, and
its size is $q$, it is necessary to include one more identity in the
basis. In this situation, the identities follow from the triple
commutator and the polynomial $x_{1}^{pq} - x_{1}^{p}$. C.
Bekh-Ochir and S. Rankin worked on this problem in 2011,
\cite{Bekh}.

Over the last fifteen years, there have been many studies on the
graded identities of the Grassmann algebra. When the ground field
has characteristic zero, Giambruno, Mishchenko and Zaicev
\cite{Giambruno1} described the $\mathbb{Z}_{2}$-graded identities
(respective $\mathbb{Z}_{2}$-graded co-dimensions) of the Grassmann
algebra equipped with canonical grading. Anisimov \cite{Anisimov}
and Da Silva \cite{Viviane2} completed the computation of the
$\mathbb{Z}_{2}$-co-dimensions of the Grassmann algebra for a basis
which $V$ was homogeneous in this grading. Da Silva and Di Vincenzo
\cite{Viviane} described the $\mathbb{Z}_{2}$-graded identities of
the Grassmann algebra for any non-trivial homogeneous
$\mathbb{Z}_{2}$-grading such that its underlying vector space is
homogeneous.

When the ground field is infinite, with a positive characteristic,
Centrone \cite{Centrone} describes the $\mathbb{Z}_{2}$-graded
identities of the Grassmann algebra in the situation explored by Da
Silva and Di Vincenzo.

For the purposes of this paper, the ground field is finite. We found
a basis for the graded polynomial identities of the Grassmann
algebra for any non-trivial $\mathbb{Z}_{2}$-grading, such that its
underlying vector space is homogeneous.

Initially, we use some results of the results of Regev \cite{Regev2}
and Ochir- Rankin \cite{Bekh}. The work of Centrone \cite{Centrone}
and Ochir-Rankin \cite{Bekh} provides the basis for the strategy we
employ to prove the main theorems. The paper by Siderov and Chiripov
\cite{Siderov} motivated our construction of the SS Total Order. In
the section 4, the work of Da Silva and Di Vincenzo \cite{Viviane}
was critical an we drew from it the majority of the computational
lemmas inside the text.

\section{Preliminaries}

In this paper, $F$ denotes a fixed finite field of characteristic
$char F = p > 2$ and size $|F| = q$. Moreover, all vector spaces and
all algebras are to be over $F$. In order to denote the elements of
$F$, we shall use small letter of Greek alphabet.

\begin{definition}
The algebra $A$ is $\mathbb{Z}_{2}$-graded when $A$ can be written
as a direct sum of subspaces $A = A_{0} \oplus A_{1}$ such that for
all $i, j \in \mathbb{Z}_{2}$, $A_{i}A_{j} \subset A_{i+j}$. The
decomposition $(A_{0},A_{1})$ is called a $\mathbb{Z}_{2}$-grading
on $A$. We shall call $A_{0}$ the even component and $A_{1}$ the odd
component. The $\mathbb{Z}_{2}$-grading $(A,0)$ is referred as
trivial. An element $a \in A$ is referred to as a homogeneous
element when $a \in A_{0}\cup A_{1}$ and we denote its
$\mathbb{Z}_{2}$-degree (when $a\neq 0$) by $\alpha(a)$.
\end{definition}

It is well known that $A$ can be graded by $\mathbb{Z}_{2}$ (in a
non-trivial way) if, and only if, $A$ admits an automorphism of
order two. If $\phi: A \rightarrow A$ is an automorphism of order
two, then:
\begin{center}
$A_{0} = \{2^{-1}(a + \phi(a))| a \in A\}$ and $A_{1} = \{2^{-1}(a -
\phi(a))| a \in A\}$.
\end{center}

Let $Y = \{y_{1},\ldots,y_{n},\ldots\}$ and $Z =
\{z_{1},\ldots,z_{n},\ldots\}$ be two countable sets of variables.
We denote by $F\langle X\rangle$ the free algebra freely generated
by $X = Y \cup Z$. For any variable $y_{i} \in Y$, we say that
$\alpha(y_{i}) = 0$; similarly for any variable $z_{i} \in Z$ we say
that $\alpha(z_{i}) = 1$. We define the $\mathbb{Z}_{2}$-degree of a
monomial $m = x_{1}\ldots x_{n} \in F\langle X\rangle$ by $\alpha(m)
= \alpha(x_{1}) + \ldots + \alpha(x_{n})$. In this way, $F\langle
X\rangle$ is a $\mathbb{Z}_{2}$-graded algebra, whereas $F\langle X
\rangle_{0}$ is spanned by the monomials of $\mathbb{Z}_{2}$-degree
$0$ and the empty word $1$, and $F\langle X\rangle_{1}$ is spanned
by the monomials of $\mathbb{Z}_{2}$-degree $1$. A polynomial
$f(x_{1},\ldots,x_{n})$ is called essential when each variable
$x_{i}, i = 1,\ldots,n,$ appears at least on time in each monomial
of $f$.

Let $A$ be a $\mathbb{Z}_{2}$-graded algebra. A polynomial
$f(x_{1},\ldots,x_{n}) \in F\langle X\rangle$ is called a
$\mathbb{Z}_{2}$-graded polynomial identity for $A$ (or a $2$-graded
polynomial identity for $A$) when $f(a_{1},\ldots,a_{n}) = 0$ for
all $a_{i} \in A_{\alpha(x_{1})}$, $i = 1,\ldots,n$. The set of all
identities of $A$ is denoted by $T_{2}(A)$. An endomorphism $\phi$
of $F\langle X\rangle$ is called a $\mathbb{Z}_{2}$-graded
endomorphism when $\phi(F\langle X \rangle_{i}) \subset F\langle X
\rangle_{i},i = 0,1$. An ideal $I \subset A$ is a
$\mathbb{Z}_{2}$-graded ideal when $I = (I\cap A_{0})\oplus (I \cap
A_{1})$. An ideal $I \subset F\langle X \rangle$ is called a
$T_{2}$-ideal when $\phi(I) \subset I$ for all $\mathbb{Z}_{2}$-
graded endomorphisms $\phi$ of $F\langle X\rangle$. It is not hard
to see that $T_{2}(A)$ is a $T_{2}$-ideal. Let $S$ be a non-empty
subset of $F\langle X \rangle$. We define the $T_{2}$-ideal
generated by $S$ as the intersection of all $T_{2}$-ideals that
contain $S$, and we denote it by $\langle S \rangle$. A polynomial
$f$ is said to be a consequence of $S$ when $f \in \langle S
\rangle$. We say that $S \subset F\langle X \rangle$ is a basis for
the $\mathbb{Z}_{2}$-graded identities of $A$ when $T_{2}(A) =
\langle S \rangle$. We know that any $T_{2}$-ideal is generated (as
$T_{2}$-ideal) by its essential polynomials. Two
$\mathbb{Z}_{2}$-graded algebras $A$ and $B$ are called isomorphic
(super-algebras) if there exists an isomorphism $\rho: A \rightarrow
B$ such that $\rho(A_{i}) \subset B_{i} \ \ i = 1,2$.

Consider $[x_{1},x_{2}]:= x_{1}x_{2} - x_{2}x_{1}$ to be the
commutator of $x_{1}$ and $x_{2}$. Inductively, we define the
\textit{left normed higher commutator} as follows:
\begin{center}
$[x_{1},\ldots,x_{n-1},x_{n}]:= [[x_{1},\ldots,x_{n-1}],x_{n}] \ \
n=3,4,\ldots$.
\end{center}

Subsequently, we use the shortened term ``commutators'' for
left-normed higher commutators.

Let $B = \bigcup_{i=1}^{\infty} B_{i}$ be a union of set of ordered
commutators, where:

\begin{flushleft}
$B_{1} = \{y_{1},y_{2},\ldots,y_{n},\ldots,z_{1},z_{2},\ldots\}$,
\end{flushleft}
\begin{flushleft}
$B_{2} =
\{[x_{1},x_{2}],[x_{1},x_{3}],\ldots,[x_{2},x_{3}],\ldots\}$,
\end{flushleft}
and
\begin{flushleft}
$B_{i} = \{[x_{1},x_{2},\ldots,x_{i}],\ldots\}$ for all $i \geq 3$.
\end{flushleft}

We know that for the Lie subalgebra $L(X)$ of $F\langle
X\rangle^{(-)}$, the Lie algebra of $F\langle X \rangle$, which is
generated by $X$, is isomorphic to the free Lie algebra with $X$ as
a set of free generators. Moreover, the following elements form a
linear basis for $F\langle X \rangle$ (we will denote this linear
basis by $Pr(X)$):
\begin{center}
$x_{i_{1}}^{a_{1}}\ldots
x_{i_{n_{1}}}^{a_{n_{1}}}[x_{j_{1}},\ldots,x_{j_{l}}]^{b_{1}}\ldots
[x_{r_{1}},\ldots,x_{r_{t}}]^{b_{n_{2}}},$
\end{center}
$a_{1},\ldots,a_{n_{1}},b_{1},\ldots,b_{n_{2}}$ are non-negative
integers,
$x_{i_{1}},\ldots,x_{i_{n_{1}}},[x_{j_{1}},\ldots,x_{j_{l}}],
\ldots,\newline [x_{r_{1}},\ldots,x_{r_{t}}] \in B$, and $x_{i_{1}}
< \ldots < x_{i_{n_{1}}} < [x_{j_{1}},\ldots,x_{j_{l}}] < \ldots <
[x_{r_{1}},\ldots,x_{r_{t}}]$ are ordered elements incrementally. In
the next definition, we present some elements of $Pr(X)$.

\begin{definition}
Let $a = (\prod_{r = 1}^{n} y_{j_{r}}^{a_{j_{r}}})(\prod_{r = 1}^{m}
z_{i_{r}}^{b_{i_{r}}})[x_{t_{1}},x_{t_{2}}]\ldots
[x_{t_{2l-1}},x_{t_{2l}}] \in Pr(X)$. We define:
\begin{enumerate}
\item $beg(a): = (\prod_{r = 1}^{n} y_{j_{r}}^{a_{j_{r}}})(\prod_{r =
1}^{m} z_{i_{r}}^{b_{i_{r}}})$ and $\psi(a) := x_{t_{1}}\ldots
x_{t_{2l}}$;
\item $pr(z)(a) = z_{i_{1}}$;
\item $Deg_{x_{i}} a$: the number of times that the variable $x_{i}$
appears in $beg(a)\psi(a)$;
\item $deg_{Y} a:= \sum_{y \in Y} Deg_{y}(a)$, $deg_{Z} a:= \sum_{z \in Z} Deg_{z}(a)$
and $deg a := deg_{Z} a + deg_{Y} a$.
\item $\mathcal{V}(a) := \{x \in X | Deg_{x}(a) > 0\}$;
\subitem $Yyn(a) := \{x \in \mathcal{V}(a) \cap Y| Deg_{x}(beg (a))
> 0, Deg_{x}(\psi(a)) = 0\}$;
\item $SS = \{u \in Pr(X)| u = (\prod_{r = 1}^{n} y_{j_{r}}^{a_{j_{r}}})(\prod_{r = 1}^{m}
z_{i_{r}}^{b_{i_{r}}})[x_{t_{1}},x_{t_{2}}]\ldots
[x_{t_{2l-1}},x_{t_{2l}}],\newline 0 \leq
a_{j_{1}},\ldots,a_{j_{n}},b_{i_{1}},\ldots,b_{i_{m}} \leq p - 1, \
\mbox{and} \ \psi(a) \ \mbox{is multilinear or} \ \psi(a) = 1\}$.
\item $SS0 = \{u \in Pr(X)| u = (\prod_{r = 1}^{n} y_{j_{r}}^{a_{j_{r}}})(\prod_{r = 1}^{m}
z_{i_{r}}^{b_{i_{r}}}), 0 \leq a_{j_{1}},\ldots,a_{j_{n}} \leq
\newline p-1, 0 \leq b_{i_{1}},\ldots,b_{i_{m}} \leq 1\}$.
\item $SS1 = \{u \in SS | Deg_{z}(u) \leq k + 1\}$.
\item $SS2 = \{u \in SS | deg_{Y}(\psi(a)) \leq k \ \mbox{and} \ deg_{Z}(beg(a)) + deg_{Y}(\psi(a)) \leq
k+1\}$.
\item $SS3 = \{u \in SS2| \mbox{if} \ \ deg_{Z} beg(u_{i}) + deg_{Y} \psi(u_{i}) = k + 1,
\newline \mbox{then} \ Deg_{pr(z)(u_{i})} \psi(u_{i}) = 0\}$.
\end{enumerate}
\end{definition}

\begin{definition}[\textbf{SS Total Order}]
Given two elements $u,v \in SS$, we say that $u < v$ when:
\begin{enumerate}
\item $deg u < deg v$ or;
\item $deg u = deg v$, but $beg(u) <_{lex-rig} beg(v)$ or;
\item $deg u = deg v, beg(u) = beg(v)$, but $\psi(u) <_{lex-rig}
\psi(v)$.
\end{enumerate}
\end{definition}
\begin{Remark}
The symbol $``lex-rig''$ denotes the right lexicographical order.
\end{Remark}

\begin{definition}
Let $f = \sum_{i = 1}^{n}\lambda_{i}u_{i}$ be a linear combination
of distinct elements of $SS$. We call the \textit{leading term} of
$f$, denoted by $LT(f)$, the element $u_{i} \in
\{u_{1},\ldots,u_{n}\}$ such that $u_{j} \leq u_{i}$ for every $j
\in \{1,\ldots,n\}$.
\end{definition}

\begin{definition}
Let $f = \sum_{i = 1}^{n}\lambda_{i}u_{i}$ be a linear combination
of distinct elements of $SS$. We call $u_{i}$ \textit{bad} if the
following conditions hold:
\begin{enumerate}
\item $Deg_{x}(u_{i}) = Deg_{x}(LT(f))$ for every $x \in X$;
\item If $deg_{Z}(beg(LT(f))) > 0$ and $z \in Z - \{pr(z)(LT(f))\}$, then $Deg_{z} beg(LT(f)) = Deg_{z}
beg(u_{i})$;
\item If $deg_{Z}(beg(LT(f))) > 0$ and $z = pr(z)(LT(f))$, then $Deg_{z}(beg(u_{i})) + 1 =
Deg_{z}(beg(LT(f)))$;
\item For every $x \in Y$, we have $Deg_{x} beg(LT(f)) \leq Deg_{x}
beg(u_{i})$.
\end{enumerate}
If $f$ has a bad term, we denote by $LBT(f)$ its even worse term.
\end{definition}
\begin{Remark}\label{observacao}
Note that if $f$ has a bad term $u_{i}$, then $u_{i} < LT(f)$.
Moreover, there exists a variable $x \in Y$ such that
$Deg_{x}(beg(LT(f))) < Deg_{x}(beg(u_{i}))$.
\end{Remark}

Let $V=\{v_1,v_2,\ldots\}$ be an infinite countable set, then we
denote by $G=G(V)$ the (unitary) Grassmann algebra generated by $V$,
i.e. $F\langle V\rangle/I$, where for each $i$, $e_i=v_i+I$ and $I$
is the ideal generated by $\{v_iv_j+v_jv_i|i,j\in\mathbb{N}\}$. It
is well known that $B_G=\left\{e_{i_1}e_{i_2}\cdots e_{i_n}\mid n\in
\mathbb{N}, {i_1}<{i_2}<\cdots<{i_n}\right\}$ is a basis of $G$ as a
vector space over $F$. We denote by $G^{*}$ the (non-unitary)
infinite dimensional Grassmann algebra, i.e., the subalgebra of $G$
generated by $\{e_{1},\ldots,e_{n},\ldots\}$.

\begin{definition}
For $a = e_{i_{1}}\ldots e_{i_{n}} \in \mathcal{B} - \{1_{G}\}$, let
$supp(a) = \{e_{i_{1}},\ldots,e_{i_{n}}\}$ (support of $a$) and
$wt(a): = |supp(a)|$, while $supp(1_{G}) = \emptyset$ and $wt(1_{G})
= 0$. Now, for any $g = \sum_{i=1}^{n} \lambda_{i}a_{i} \in G -
\{0\}$ (where $a_{i} \in \mathcal{B}$ and $\lambda_{i} \in F -
\{0\}$). Let $supp(g) := \cup_{i=1}^{n}supp(a_{i})$ (support of $g$)
and $wt(g):= max\{wt(a_{i})| i = 1,\ldots,n\}$ (support-length of
$g$) and $dom(g) := \sum_{wt(a_{i}) = wt(g)} \lambda_{i}a_{i}$
(dominant part of $g$), while we define $supp(0) = \emptyset$ and
$|supp(0)| = 0$.
\end{definition}

We can prove the next proposition by routine calculations.

\begin{proposition}\label{fundamental}
Let $f = [x_{1},x_{2}]\ldots[x_{2n - 1},x_{2n}]$ and $g = x_{1}^{k},
k < p,$ be polynomials. Then, the following assertions hold:
\begin{enumerate}
\item $f(e_{1},e_{2},\ldots,e_{2n}) = 2^{n} \prod_{i = 1}^{2n} e_{i}$
\item $f(e_{1} + e_{2}e_{3}, e_{4} + e_{5}e_{6}, \ldots, e_{1 +
3(n-1)} + e_{2 + 3(n-1)}e_{3n}) = 2^{n}\prod_{i = 1}^{n}e_{1 +
3(i-1)}$ \item $g(e_{1}e_{2} + \ldots + e_{2k - 1}e_{2k}) =
k!\prod_{i = 1}^{2k}e_{i}$
\item $dom(g(e_{1}e_{2} + \ldots + e_{2k - 1}e_{2k} + e_{2k+1})) =
k!\prod_{i = 1}^{2k}e_{i}$
\end{enumerate}
\end{proposition}

Consider the following automorphisms of order $2$ on $G$:

$$
\left\{\begin{array}{l}
\phi_{0}: G \rightarrow G \\
e_{i} \mapsto  - e_{i} , i = 1,2,\ldots,n,\ldots\\
\end{array}\right.
$$

$$
\left\{\begin{array}{l}
\phi_{\infty}: G \rightarrow G \\
e_{i} \mapsto  e_{i} , \ \mbox{if} \ i \  \mbox{is even} \\
e_{i} \mapsto - e_{i}, \ \mbox{if} \ i \  \mbox{is odd}
\end{array}\right.
$$

$$
\left\{\begin{array}{l}
\phi_{k^{*}}: G \rightarrow G \\
e_{i} \mapsto - e_{i} , i = 1,\ldots,k \\
e_{i} \mapsto e_{i}, i = k+1, \ldots,
\end{array}\right.
$$

$$
\left\{\begin{array}{l}
\phi_{k}: G \rightarrow G \\
e_{i} \mapsto  e_{i} , i = 1,\ldots,k \\
e_{i} \mapsto - e_{i}, i = k+1, \ldots,
\end{array}\right.
$$

Each of these four automorphisms induces a non-trivial
$\mathbb{Z}_{2}$-grading on $G$. As such, from here, $G$ has a
grading induced by one of the four automorphisms reported above.

Amitai Regev, Chuluundorj Bekh-Ochir and Stuart Rankin contributed
our understanding of the Grassmann algebra over a finite field.
Here, we report some o their results.

\begin{lemma}[Regev,\cite{Regev2},Lemma 1.2]
If $char F = p > 0$, then $G$ and $G^{*}$ satisfy the ordinary
identity $\sum_{\sigma \in S_{p}} sgn(\sigma)x_{\sigma(1)}\ldots
x_{\sigma(p)}$, where $S_{p}$ is the permutation group of $p$
elements. $G^{*}$ satisfies the ordinary identity $x^{p}$.
\end{lemma}

\begin{lemma}[Regev, \cite{Regev2}, Corollary 1.4]
If $\overline{a} = \alpha1_{G} + a$, where $\alpha \in F, a \in
G^{*}$, then ${\overline{a}}^{p} = (\alpha1_{G} + a)^{p} =
\alpha^{p}1_{G} + a^{p} = \alpha^{p}1_{G}$.
\end{lemma}

\begin{lemma}[Regev, \cite{Regev2}, Corollary 1.4]
Let $char(F) = p, |F| = q = p^{t} < \infty$, then
\begin{enumerate}
\item In addition to the ordinary identity $[[x,y],z]$, $G$ satisfies
the identity $x^{pq} - x^{p}$.
\end{enumerate}
\end{lemma}

An ordinary polynomial $f(x_{1},\ldots,x_{n}) = \sum_{j =
1}^{l}\lambda_{j}m_{j} \in Pr(X)$, where $\psi(m_{1}) = \ldots =
\psi(m_{l}) = 1$, is said to be a $p$-polynomial when for any $j \in
\{1,\ldots,l\}$ and $i \in \{1,\ldots,n\}$, we have
$Deg_{x_{i}}m_{j} \equiv \ \ 0 \ \ mod \ \ p$ and $Deg_{x_{i}}m_{j}
< pq$.

\begin{corollary}[Ochir-Rankin, \cite{Bekh}, Corollary 3.1]
Let $f$ be a $p$-polynomial. If $f$ is an ordinary polynomial
identity of $G$, then $f$ is the zero polynomial.
\end{corollary}

In this paper, we consider $\mathbb{Z}_{2}$-graded $p$-polynomial in
$Y$.

Keeping in mind Lemma 1.1.4 of \cite{Drensky-Formanek}, we have
$[x_{1},x_{2}][x_{3},x_{4}] - [x_{1},x_{3}][x_{2},x_{4}] \in \langle
[x_{1},x_{2},x_{3}] \rangle_{T_{2}}$. Note also that
$[x_{2},x_{1},\ldots,x_{1}],$ where $x_{1}$ appears $p_{1}$ times in
the brackets of the last commutator, follows from
$[x_{1},x_{2},x_{3}]$. If $char F = p$, we have $-
[x_{2},x_{1},\ldots,x_{1}] = [x_{1}^{p},x_{2}]$. Hence,
$[x_{1}^{p},x_{2}]$ follows from $[x_{1},x_{2},x_{3}]$ when $char F
= p > 2$ (which is a well know fact). The next proposition is a
synthesis of these results and we adapt this synthesis for the
$\mathbb{Z}_{2}$-graded case.

\begin{proposition}\label{fundamental2}
Let $F$ be a finite field of characteristic $char F = p > 2$ and
size $|F| = q$. Let $G$ be the infinite dimensional unitary
Grassmann algebra over $F$. The following assertions then hold:
\begin{enumerate}
\item The polynomials $[x_{1},x_{2},x_{3}],y_{1}^{pq} -
y_{1}^{p}$, and $z_{1}^{p}$ are $\mathbb{Z}_{2}$-polynomial
identities for $G$.
\item Let $f(y_{1},\ldots,y_{n})$ be a $p$-polynomial. If $f \notin
T_{2}(G)$, then there exist $\alpha_{1},\ldots,\alpha_{n} \in F$
such that $f(\alpha_{1}1_{G},\ldots,\alpha_{n}1_{G}) \neq 0$.
\item $[x_{1},x_{2}][x_{3},x_{4}] - [x_{1},x_{3}][x_{2},x_{4}], [x_{1}^{p},x_{2}] \in \langle
[x_{1},x_{2},x_{3}]\rangle_{T_{2}}$.
\item Let $f = \sum_{i=1}^{n}\lambda_{i}v_{i}$ be a linear combination
from $Pr(X)$. \subitem  For modulo $\langle [x_{1},x_{2},x_{3}],
z_{1}^{p}, y_{1}^{pq} - y_{1}^{p} \rangle$, $f$ can be written as
\begin{center} $\sum_{i=1}^{m}f_{i}u_{i}$, \end{center} where
$f_{1},\ldots,f_{m}$ are $p$-polynomials and $u_{1},\ldots,u_{m} \in
SS$ is (are) distinct.
\end{enumerate}
\end{proposition}

Now, we present some $\mathbb{Z}_{2}$-graded identities of
$G_{can},G_{\infty},G_{k^{*}}$ and $G_{k}$, which include the
following folkloric results: $[y_{1},y_{2}],[y_{1},z_{2}],z_{1}z_{2}
+ z_{2}z_{1} \in T_{2}(G_{can}); [x_{1},x_{2},x_{3}] \in
T_{2}(G_{\infty}); [x_{1},x_{2},x_{3}], z_{1}\ldots z_{k+1} \in
T_{2}(G_{k^{*}})$.

Let $T' = (i_{1},\ldots,i_{l})$ and $T = (j_{1},\ldots,j_{t})$ be
two strictly ordered sequences of positive integers such that $t$ is
even, $l + t = m$, and $\{1,\ldots,m\} =
\{i_{1},\ldots,i_{l},j_{1},\ldots,j_{t}\}$. Let us next define:
\begin{center}
$f_{T}(z_{1},\ldots,z_{m}) = z_{i_{1}}\ldots
z_{i_{l}}[z_{j_{1}},z_{j_{2}}]\ldots[z_{j_{t-1}},z_{j_{t}}]$.
\end{center}
In the same way, let $T' = (i_{1},\ldots,i_{l})$ and $T =
(j_{1},\ldots,j_{t})$ be two strictly ordered sequences of positive
integers such that $t$ is odd, $l + t = m$, and $\{1,\ldots,m\} =
\{i_{1},\ldots,i_{l},j_{1},\ldots,j_{t}\}$. Let us then define:
\begin{center}
$r_{T}(y_{1},z_{1},\ldots,z_{m}) = z_{i_{1}}\ldots
z_{i_{l}}[y_{1},z_{j_{1}}]\ldots[z_{j_{t-1}},z_{j_{t}}]$.
\end{center}
\begin{definition}
Let $m \geq 2$. Let:
\begin{center}
$g_{m}(z_{1},\ldots,z_{m}) = \sum\limits_{|T| \ \mbox{is even}}
(-2)^{\frac{-|T|}{2}}f_{T}(z_{1},\ldots,z_{m})$.
\end{center}
Moreover: $g_{1}(z) = z$.
\end{definition}

In the next proposition, we invoke a result of Lucio Centrone
(Theorem 7.1,\cite{Centrone}), and a result of Onofrio Di Vincenzo
and Viviane Tomaz da Silva (Theorem 38, \cite{Viviane}). This
proposition is an immediate consequence of these results.

\begin{proposition}\label{7.1}
Let $F$ be an infinite of characteristic $char F \neq 2$. The
following polynomials are $\mathbb{Z}_{2}$-graded identities for
$G_{k}$:
\begin{enumerate}
\item $[y_{1},y_{2}]\ldots [y_{k},y_{k+1}]$ (if $k$ is odd) \ \ (1);
\item $[y_{1},y_{2}]\ldots [y_{k-1},y_{k}][y_{k+1},x]$ (if $k$
is even and $x \in X - \{y_{1},\ldots,y_{k+1}\})$ \ \ (2);
\item $g_{k-l+2}(z_{1},\ldots,z_{k-l+2})[y_{1},y_{2}]\ldots
[y_{l-1},y_{l}]$ (if $l \leq k$ and $l$ is even) \ \ (3);
\item $g_{k-l+2}(z_{1},\ldots,z_{k-l+2})[z_{k-l+3},y_{1}][y_{2},y_{3}]\ldots[y_{l-1},y_{l}]$
(if $l \leq k$ and $l$ is odd) \ \ (4);
\item
$[g_{k-l+2}(z_{1},\ldots,z_{k-l+2}),y_{1}]\ldots[y_{l-1},y_{l}]$ (if
$l\leq k$ and $l$ is odd) \ \ (5);
\item $[x_{1},x_{2},x_{3}]$ \ \ (6).
\end{enumerate}
\end{proposition}
\begin{Remark}
Let $F$ be a finite field of $char F \neq 2$. It is well known that
the six types of $\mathbb{Z}_{2}$-graded polynomials reported above
are also $\mathbb{Z}_{2}$-graded polynomial identities for $G_{k}$.
\end{Remark}

The following two corollaries are immediate consequences of
Proposition \ref{7.1}.

\begin{corollary}\label{gte}
Let $I$ be the $T_{2}$-ideal generated by the graded identities of
type $(3)$. In the free super-algebra $F \langle X\rangle$, we have:
\begin{center}
$z_{1}z_{2}\ldots z_{k - l + 2}[y_{1},y_{2}]\ldots [y_{l-1},y_{l}]
\equiv \ a.b \ \ mod \ \ I$,
\end{center}
where $l \leq k$, $l$ is even, and
\begin{enumerate}
\item $a(z_{1},\ldots,z_{k-l+2}) = (\sum_{|T| \mbox{is even and non-empty}}
-(-2)^{-\frac{|T|}{2}} f_{T}(z_{1},\ldots,z_{k-l+2}))$;
\item $b(y_{1},\ldots,y_{l}) = [y_{1},y_{2}]\ldots [y_{l-1},y_{l}]$.
\end{enumerate}
\end{corollary}

\begin{corollary}\label{gte2}
Let $I$ be the $T_{2}$-ideal generated by the graded identities of
type $(4)$. In the free super-algebra $F \langle X\rangle$, we have:
\begin{center}
$z_{1}z_{2}\ldots z_{k - l + 2}[z_{k - l + 3},y_{1}]\ldots
[y_{l-1},y_{l}] \equiv a.b \ \ mod \ \ I$,
\end{center}
where $l \leq k$, $l$ is odd, and
\begin{enumerate}
\item $a(z_{1},\ldots,z_{k-l+2}) = (\sum_{|T| \mbox{is even and
non-empty}} -(-2)^{-\frac{|T|}{2}} f_{T}(z_{1},\ldots,z_{k-l+2}))$;
\item $b(z_{k-l+3},y_{1},\ldots,y_{l}) = [z_{k-l+3},y_{1}]\ldots
[y_{l-1},y_{l}]$.
\end{enumerate}
\end{corollary}

From now on, the polynomials of $F\langle X \rangle$ will be written
as linear combination of elements from $Pr(X)$.

\section{$G_{can}, G_{\infty}, G_{k^{*}} \ (k \geq 0)$}

In this section, we describe the $\mathbb{Z}_{2}$-graded identities
of $G_{can}, G_{\infty}$ and $G_{k^{*}}, k \geq 0$. Recall that the
ordinary identities of $G$ were described in \cite{Bekh}, and the
authors proved the following result a few years ago.

\begin{theorem}[Ochir-Rankin,Theorem 3.1,\cite{Bekh}]
The ordinary polynomial identities of $G$ follow from
\begin{center}
$[x_{1},x_{2},x_{3}]$ and $x_{1}^{pq} - x_{1}^{p}$.
\end{center}
\end{theorem}

\begin{definition}
We denote by $I_{1}$ the $T_{2}$-ideal generated by polynomials
$[y_{1},y_{2}],[y_{1},z_{2}],\newline z_{1}z_{2} + z_{2}z_{1}$ and
$y_{1}^{pq} - y_{1}^{p}$.
\end{definition}

\begin{theorem}\label{1}
$T_{2}(G_{can}) = I_{1}$.
\end{theorem}
\begin{proof}
Suppose the assertion of the theorem is false. So, there exists an
essential polynomial:
\begin{center}
$f(y_{1},\ldots,y_{n},z_{1},\ldots,z_{m}) \in T_{2}(G_{can}) -
I_{1}$.
\end{center}
Due to the four identities that generate $I_{1}$ and Proposition
\ref{fundamental2}, we may assume that $f = \sum_{i=1}^{l}
f_{i}M_{i}$, where each $f_{i}$ is a non-zero $p$-polynomial and
$M_{i} \in SS0$. Moreover, $M_{i} \neq M_{j}$ whenever $i \neq j$.

Let $M_{i}$ be the greatest element of $\{M_{1},\ldots,M_{l}\}$. For
simplicity's sake, we assume that $Yyn(M_{i}) =
\{y_{1},\ldots,y_{n}\}$.

There exists an $n$-tuple $(\alpha_{1}1_{G},\ldots,\alpha_{n}1_{G})$
such that $f_{i}(\alpha_{1}1_{G},\ldots,\alpha_{n}1_{G}) \neq 0$.
Moreover
\begin{center}
$f_{i}(\alpha_{1}1_{G} + e_{1}e_{2},\ldots,\alpha_{n}1_{G} +
e_{2n-1}e_{2n})M_{i}(\alpha_{1}1_{G} +
e_{1}e_{2},\ldots,\alpha_{n}1_{G} + e_{2n-1}e_{2n},e_{2n +
1},\ldots,e_{2n + m}) \neq 0$
\end{center}
and
\begin{center}
$wt(M_{i}(\alpha_{1}1_{G} + e_{1}e_{2},\ldots,\alpha_{n}1_{G} +
e_{2n-1}e_{2n},e_{2n + 1},\ldots,e_{2n + m})) = \lambda e_{1}\ldots
e_{2n}e_{2n + 1}\ldots e_{2n + m}$
\end{center}
for some non-zero $\lambda \in F$.

Note that if $M_{j} \in \{M_{1},\ldots,\widehat{M_{i}},\ldots
M_{n}\}$ (the ``hat'' over a monomial means that it can be missing),
then the support of no summand of $M_{j}(\alpha_{1}1_{G} +
e_{1}e_{2},\ldots,\alpha_{n}1_{G} + e_{2n-1}e_{2n},e_{2n +
1},\ldots,e_{2n + m})$ contains more than $2n + m - 1$ elements.
Thus, $wt(f(\alpha_{1}1_{G} + e_{1}e_{2},\ldots,\alpha_{n}1_{G} +
e_{2n-1}e_{2n},e_{2n + 1},\ldots,e_{2n + m})) = \lambda e_{1}\ldots
e_{2n}\ldots e_{2n + m}$. This is a contradiction.
\end{proof}

\begin{definition}
We denote by $I_{2}$ the $T_{2}$-ideal generated by polynomials
$[x_{1},x_{2},x_{3}],z_{1}^{p}$ and $y_{1}^{pq} - y_{1}^{p}$.
\end{definition}



In Theorems \ref{I2}, \ref{I3}, and \ref{gfinal} (Cases 1 and 2), we
suppose without loss of generality that:
\begin{center}
$LT(f) = y_{1}^{a_{1}}\ldots y_{n_{1}}^{a_{n_{1}}}
y_{n_{1}+1}^{a_{n_{1} + 1}}\ldots
y_{n_{2}}^{a_{n_{2}}}z_{1}^{b_{1}}\ldots z_{m_{1}}^{b_{m_{1}}}
z_{m_{1}+1}^{b_{m_{1} + 1}}\ldots
z_{m_{2}}^{b_{m_{2}}}\newline[y_{n_{1}+1},y_{n_{1}+2}]\ldots
[y_{n_{2}},y_{n_{2}+1}]\ldots [y_{l_{1}},z_{m_{1}+1}]\ldots
[z_{m_{2}-1},z_{m_{2}}]\ldots [z_{l_{2}-1},z_{l_{2}}],$
\end{center}
where $n_{1} < n_{2} < l_{1}, m_{1} < m_{2} < l_{2};
a_{1},\ldots,a_{n_{2}},b_{1},\ldots,b_{m_{2}} > 0$ and \newline $f =
f(y_{1},\ldots,y_{l_{1}},z_{1},\ldots,z_{l_{2}})$. We denote the set
$\{y_{1},\ldots,y_{l_{1}},z_{1},\ldots,z_{l_{2}}\}$ by
$Variable(f)$. In Theorem \ref{gfinal} (Case 3), we suppose that
$LBT(f)$ has the same expression as that above.

\begin{theorem}\label{I2}
$T_{2}(G_{\infty}) = I_{2}$.
\end{theorem}
\begin{proof}
The proof of this theorem is similar to that of Theorem \ref{1}. We
suppose that there exists an essential polynomial $f \in
T_{2}(G_{\infty}) - I_{2}$. Then, we get a contradiction.

In fact, suppose that there exists
$f(y_{1},\ldots,y_{l_{1}},z_{1},\ldots,z_{l_{2}}) \in
T_{2}(G_{\infty}) - I_{2}$. By Proposition \ref{fundamental2} we may
assume $\sum_{i = 1}^{l} f_{i}M_{i}$, where each $f_{i}$ is a
$p$-polynomial and $M_{i} \in SS$. Moreover, $M_{i} \neq M_{j}$ if
$i \neq j$. It is convenient to take $f_{1} = \ldots = f_{l} = 1$ to
avoid repetitive arguments.

Consider the following graded homomorphism of $F\langle
y_{1},\ldots,y_{l_{1}},z_{1},\ldots,z_{l_{2}} \rangle $:
\begin{flushleft}

$\phi: \{y_{1},\ldots,y_{l_{1}},z_{1},\ldots,z_{l_{2}}\} \rightarrow
G
\newline$

$y_{1} \mapsto \sum_{l=1}^{a_{1}}e_{4l - 2}e_{4l} \newline \ldots
\newline
y_{n_{1}} \mapsto \sum_{l = a_{1} + \ldots + a_{n_{1}-1} + 1}^{a_{1}
+ \ldots + a_{n_{1}}}e_{4l - 2}e_{4l} \newline y_{n_{1}+1} \mapsto
e_{4(\sum_{l=1}^{n_{1}}a_{l}) + 2} + \sum_{l=1}^{a_{n_{1} +
1}}e_{4(\sum_{l=1}^{n_{1}}a_{l})+ 4l}e_{4(\sum_{l=1}^{n_{1}}a_{l})+
4l + 2} \newline \ldots \newline y_{n_{2}} \mapsto
e_{4(\sum_{i=1}^{n_{2}-1}a_{i}) + 2(n_{2}-n_{1})} + \newline +
\sum_{l = 1}^{a_{n_{2}}}e_{ 4(\sum_{i=1}^{n_{2}-1}a_{i}) +
2(n_{2}-n_{1}) + 4l - 2}e_{ 4(\sum_{i=1}^{n_{2}-1}a_{i}) +
2(n_{2}-n_{1})+ 4l} \newline y_{n_{2} + 1} \mapsto
e_{4(\sum_{i=1}^{n_{2}}a_{i}) + 2(n_{2}-n_{1}+1)} \newline \ldots
\newline y_{l_{1}} \mapsto e_{4(\sum_{i=1}^{n_{2}}a_{i}) + 2(l_{1}-n_{1})}
\newline z_{1} \mapsto \sum_{l = 1}^{b_{1}}e_{2l-1}e_{M + 2l}
\newline \ldots \newline z_{m_{1}} \mapsto \sum_{l = b_{1} + \ldots + b_{m_{1}-1}+1}^{b_{1}
+ \ldots + b_{m_{1}}}e_{2l - 1}e_{M + 2l} \newline z_{m_{1}+1}
\mapsto e_{2(\sum_{l=1}^{m_{1}}b_{l}) + 1} + \sum_{l=b_{1} + \ldots
+ b_{m_{1}}+1}^{b_{1} + \ldots + b_{m_{1}+1}} e_{2l + 1}e_{M + 2l}
\newline \ldots \newline z_{m_{2}} \mapsto e_{2(\sum_{l=1}^{m_{2}-1}b_{l}) + 2(m_{2}-m_{1})
- 1} + \sum_{l = b_{1} + \ldots + b_{m_{2}-1}+1}^{b_{1} + \ldots +
b_{m_{2}}} e_{2(l + m_{2} - m_{1}) - 1}e_{M + 2l} \newline z_{m_{2}
+ 1} \mapsto e_{2(\sum_{l=1}^{m_{2}}b_{l} + m_{2}-m_{1} ) + 1}
\newline \ldots \newline z_{l_{2}} \mapsto e_{2(\sum_{l=1}^{m_{2}}b_{l}) + 2(l_{2}-m_{1}) -
1}$, where $M = 4(\sum_{i=1}^{n_{2}}a_{i}) + 2(l_{1}-n_{1})$.
\end{flushleft}

All the summands of $\phi(LT(f)) - dom(\phi(LT(f)))$ have a support
with less than $\frac{M + 2l}{2} + \sum_{l=1}^{m_{2}}b_{l} + (l_{2}
- m_{1})$ elements. According to Proposition \ref{fundamental},
\begin{center}
$dom(\phi(LT(f))) = \lambda(e_{1}\ldots
e_{2(\sum_{l=1}^{m_{2}}b_{l}) + 2(l_{2}-m_{1}) - 1})(e_{2}\ldots
e_{M + 2l})$
\end{center}
for some non-zero $\lambda \in F$.


If $u_{i} \neq LT(f)$, one of two things can occur: 1) There exists
$x \in X$ such that $Deg_{x} u_{i} < Deg_{x} LT(f)$. In this
situation, no summand of $\phi(u_{i})$ contains $supp(\phi(x))$. 2)
For all $x \in X$, we have $Deg_{x}(u_{i}) = Deg_{x}(LT(f))$. Now,
there must exist $x \in X$ such that $Deg_{x}(beg(u_{i})) <
Deg_{x}(beg(LT(f)))$. So, $\phi(u_{i}) = 0$.

We conclude that $dom(\phi(f)) = dom(\phi(LT(f))) \neq 0$. This is a
contradiction and completes the proof.

\end{proof}

\begin{definition}
We denote by $I_{3}$ the $T_{2}$-ideal generated by polynomials
$[x_{1},x_{2},x_{3}],\newline z_{1}.\ldots.z_{k+1},z_{1}^{p}$ and
$y_{1}^{pq} - y_{1}^{p}$.
\end{definition}

\begin{theorem}\label{I3}
$T_{2}(G_{k}) = I_{3}$.
\end{theorem}
\begin{proof}
The proof of this theorem follows word by word the proof of Theorem
\ref{I2}.

Therefore, there exists $f \in \sum_{i = 1}^{l}f_{i}M_{i}$ where
$f_{i}$ is a non-zero $p$-polynomial, $M_{i} \in SS1$ and $M_{i}
\neq M_{j}$ when $i \neq j$. Without loss of generality, suppose
that $f_{1} = \ldots = f_{l} = 1$. Now, consider the following
graded homomorphism of $F\langle
y_{1},\ldots,y_{l_{1}},z_{1},\ldots,z_{l_{2}}\rangle$:
\begin{flushleft}
$\phi: \{y_{1},\ldots,y_{l_{1}},z_{1},\ldots,z_{l_{2}}\} \rightarrow
G \newline$ $y_{1} \mapsto  \sum_{l=1}^{a_{1}}e_{k + 2l - 1}e_{k +
2l}
\newline \ldots \newline y_{n_{1}} \mapsto \sum_{l = a_{1} + \ldots
+ a_{n_{1}-1} + 1}^{a_{1} + \ldots + a_{n_{1}}}e_{k + 2l - 1}e_{k +
2l} \newline y_{n_{1}+1} \mapsto e_{k + 2(a_{1} + \ldots +
a_{n_{1}}) + 1} + \sum_{l=1}^{a_{n_{1}+1}}e_{k + 2(a_{1} + \ldots +
a_{n_{1}}) + 2l}e_{k + 2(a_{1} + \ldots + a_{n_{1}}) + 2l +
1}\newline \ldots \newline y_{n_{2}} \mapsto e_{k + 2(a_{1} + \ldots
+ a_{n_{2}-1}) + (n_{2}-n_{1})} + \newline
\sum_{l=1}^{a_{n_{2}}}e_{k + 2(a_{1} + \ldots + a_{n_{2}-1}) +
(n_{2}-n_{1}) + 2l - 1}e_{k + 2(a_{1} + \ldots + a_{n_{2}-1}) +
(n_{2}-n_{1}) + 2l} \newline y_{n_{2}+1} \mapsto e_{k + 2(a_{1} +
\ldots + a_{n_{2}-1} + a_{n_{2}}) + (n_{2}-n_{1}) + 1} \newline
\ldots \newline y_{l_{1}} \mapsto e_{k + 2(a_{1} + \ldots +
a_{n_{2}-1} + a_{n_{2}}) + (l_{1}-n_{1})}
\newline z_{1} \mapsto \sum_{l=1}^{b_{1}}e_{l}e_{Q + l} \newline
\ldots \newline z_{m_{1}} \mapsto \sum_{l = b_{1} + \ldots +
b_{m_{1}-1} + 1}^{b_{1} + \ldots + b_{m_{1}}}e_{l}e_{Q + l} \newline
z_{m_{1}+1} \mapsto e_{b_{1} + \ldots + b_{m_{1}} + 1} + \sum_{l =
b_{1} + \ldots + b_{m_{1}} + 1}^{b_{1} + \ldots + b_{m_{1}+1}} e_{l
+ 1}e_{Q+l}
\newline \ldots \newline z_{m_{2}} \mapsto e_{b_{1} + \ldots + b_{m_{2}-1} + (m_{2}-m_{1})}
+ \sum_{l = 1}^{b_{m_{2}}}e_{l + b_{1} + \ldots + b_{m_{2}-1} +
(m_{2}-m_{1})}e_{Q + b_{1} + \ldots + b_{m_{2}-1} + l} \newline
z_{m_{2} + 1} \mapsto e_{T+1} \newline \ldots \newline z_{l_{2}}
\mapsto e_{T + l_{2} - m_{2}}$,
\end{flushleft}
where $Q = k + 2(a_{1} + \ldots + a_{n_{2}-1} + a_{n_{2}}) +
(l_{1}-n_{1})$ and $T = b_{1} + \ldots + b_{m_{2}-1} + b_{m_{2}} +
(m_{2}-m_{1})$.

As in the proof of Theorem \ref{I2}, we conclude that $dom(\phi(f))
= dom(LT(f)) \neq 0$. This is a contradiction and the proof is
complete.
\end{proof}

\section{$G_{k} \ (k \geq 1)$}

In this section, we describe the $\mathbb{Z}_{2}$-graded identities
for $G_{k}$. Unlike papers \cite{Centrone} and \cite{Viviane}, we do
not a use representation theory methodology. In the next definition,
we recall the eight types identities of $G_{k}$.

\begin{definition}
We denote the $T_{2}$-ideals generated by the following eight types
identities below as $I_{4}$.
\begin{enumerate}
\item $[y_{1},y_{2}]\ldots [y_{k},y_{k+1}]$ (if $k$ is odd) \ \ (1);
\item $[y_{1},y_{2}]\ldots [y_{k-1},y_{k}][y_{k+1},x]$ (if $k$
is even and $x \in X - \{y_{1},\ldots,y_{k+1}\})$ \ \ (2);
\item $g_{k-l+2}(z_{1},\ldots,z_{k-l+2})[y_{1},y_{2}]\ldots
[y_{l-1},y_{l}]$ (if $l \leq k$ and $l$ is even) \ \ (3);
\item $g_{k-l+2}(z_{1},\ldots,z_{k-l+2})[z_{k-l+3},y_{1}][y_{2},y_{3}]\ldots[y_{l-1},y_{l}]$
(if $l \leq k$ and $l$ is odd) \ \ (4);
\item
$[g_{k-l+2}(z_{1},\ldots,z_{k-l+2}),y_{1}]\ldots[y_{l-1},y_{l}]$ (if
$l\leq k$ and $l$ is odd) \ \ (5);
\item $[x_{1},x_{2},x_{3}]$ \ \ (6);
\item $z_{1}^{p}$ \ \ (7);
\item $y_{1}^{pq} - y_{1}^{p}$ \ \ (8).
\end{enumerate}
\end{definition}

Before the proof the main main theorem, we have the following lemma.

\begin{lemma}\label{gte3}
The following assertions hold
\begin{enumerate}
\item (A1). Let $u$ be an element of $SS$ with the following property: $deg_{Z}(beg(u)) + deg_{Y}(\psi(u)) \geq k + 2$ or $deg_{Y}(\psi(u))
= k + 1$. For modulo $I_{4}$, $u$ can be written as a linear
combination of $SS2$.

\item (A2). In the free super-algebra $F
\langle X \rangle$, we have: \begin{center} $z_{2}\ldots z_{k - l +
2}[z_{1},z_{k-l+3}][y_{1},y_{2}]\ldots[y_{l-1},y_{l}] \equiv
(\sum_{J} \beta_{J} f_{J})[y_{1},y_{2}]\ldots[y_{l-1},y_{l}] \ mod \
\ I_{4}$\end{center} (if $l\leq k$ and $l$ is even) for some
$\beta_{J} \in F$, $J \subseteq \{1,\ldots,k-l+3\}$. Moreover, if
$|J| = 2$, then $1 \notin J$ and $\beta_{J} = - 1$.

\item (A3). If $v \in SS2, deg_{Z} beg(v) + deg_{Y} \psi(v) = k + 1, 2 \mid
deg_{Y}\psi(v), \mbox{and} \newline Deg_{pr(z)v} \psi(v) = 1$, then:
$v \equiv \sum_{i=1}^{n}\lambda_{i}v_{i} \ \ mod \ \ I_{4}$ where $v
- \sum_{i=1}^{n}\lambda_{i}v_{i}$ is a multihomogeneous polynomial,
and $v_{1},\ldots,v_{n} \in SS3$.

\item (A4). In the free super-algebra $F
\langle X \rangle$, we have: \begin{center} $z_{2}\ldots
z_{k-l+2}[z_{1},y_{1}][y_{2},y_{3}]\ldots[y_{l-1},y_{l}] \equiv
(\sum_{J} \beta_{J}
r_{J}(z_{1},\ldots,z_{k-l+2},y_{1}))[y_{2},y_{3}]\ldots[y_{l-1},y_{l}]
\ mod \ \ I_{4}$\end{center} (if $l\leq k$ and $l$ is odd) for some
$\beta_{J} \in F$, $J \subseteq \{1,\ldots,k-l+2\}$. Moreover, if
$|J| = 1$, then $1 \notin J$ and $\beta_{J} = 1$.

\item (A5). If $v \in SS2, deg_{Z} beg(v) + deg_{Y} \psi(v) = k + 1, 2
\nmid deg_{Y}\psi(v), \newline \mbox{and} \ Deg_{pr(z)v} \psi(v) =
1$, then: $v \equiv \sum_{i=1}^{n}\lambda_{i}v_{i} \ \ mod \ \
I_{4}$, where $v - \sum_{i=1}^{n}\lambda_{i}v_{i}$ is a
multihomogeneous polynomial, and $v_{1},\ldots,v_{n} \in SS3$.

\end{enumerate}
\end{lemma}
\begin{proof}

(A1). First, note that if $deg_{Y}(\psi(u)) > k$, then $u$ is a
consequence of $(1)$ or $(2)$. In this way, we may assume that
$deg_{Y}\psi(u) \leq k$.

We may suppose without loss of generality that:

\begin{center}
$u = y_{1}^{a_{1}}\ldots y_{n_{1}}^{a_{n_{1}}} y_{n_{1}+1}^{a_{n_{1}
+ 1}}\ldots y_{n_{2}}^{a_{n_{2}}}z_{1}^{b_{1}}\ldots
z_{m_{1}}^{b_{n_{1}}} z_{m_{1}+1}^{b_{m_{1} + 1}}\ldots
z_{m_{2}}^{b_{m_{2}}}[y_{n_{1}+1},y_{n_{1}+2}]\newline
\ldots[y_{n_{1}+l},z_{m_{1}+1}]\ldots
[z_{m_{2}-1},z_{m_{2}}][z_{m_{2}+1},z_{m_{2}+2}]\ldots
[z_{l_{2}-1},z_{l_{2}}]$,
\end{center}
where $m_{1} < m_{2} < l_{2}, n_{1} < n_{2} < n_{1} + l;
b_{1},\ldots,b_{m_{2}},a_{1},\ldots,a_{n_{2}} > 0$ and
$deg_{Z}(beg(u)) + deg_{Y}(\psi(u)) = k + 2$.

Thus, according to Corollary \ref{gte2}:
\begin{center}
$u \equiv a.b.c \ \  mod \ \ I_{4}$,
\end{center}
where

\begin{center}
$a(y_{1},\ldots,y_{n_{2}}) = y_{1}^{a_{1}}\ldots
y_{n_{1}}^{a_{n_{1}}} y_{n_{1}+1}^{a_{n_{1} + 1}}\ldots
y_{n_{2}}^{a_{n_{2}}}$; \\

$b(z_{1},\ldots,z_{m_{2}}) = (\sum_{|T| \mbox{is even and
non-empty}} -(-2)^{-\frac{|T|}{2}}
f_{T}(z_{1},\ldots,z_{m_{2}}))$;\\

$c(y_{n_{1}+1},\ldots,y_{n_{1}+l},z_{m_{1}+1},\ldots,z_{l_{2}}) =
[y_{n_{1}+1},y_{n_{1}+2}]\ldots[y_{n_{1} + l},z_{m_{1}+1}]\ldots
\newline[z_{m_{2}-1},z_{m_{2}}][z_{m_{2}+1},z_{m_{2}+2}]\ldots
[z_{l_{2}-1},z_{l_{2}}]$.
\end{center}

Then, after applying the graded identity $[x_{1},x_{2},x_{3}]$ to
$b.c$, we are done.

When $deg_{Z}(beg(u)) + deg_{Y}(\psi(u)) > k+2$, the proof is
similar by inductive arguments. To arrive at this situation, we must
replace $a$ by the following:
\begin{center}
$y_{1}^{a_{1}}\ldots y_{n_{1}}^{a_{n_{1}}} y_{n_{1}+1}^{a_{n_{1} +
1}}\ldots y_{n_{2}}^{a_{n_{2}}}z_{1}^{b_{1}-c}\ldots
z_{k_{1}}^{b_{k_{1}}}$,
\end{center}
where $k_{1} \leq m_{2}, b_{k_{1}} - c \geq 0$, and $b_{1} + \ldots
+ b_{k_{1}} - c = deg_{Z} beg(u) - (k - l + 2)$.

\vspace{0.5cm}

For (A2) and (A4), we use some of the arguments of Lemma 20-b in
\cite{Viviane}.

\vspace{0.5cm}

(A2).  First, note that
$[g_{k-l+2}(z_{1},\ldots,z_{k-l+2})[y_{1},y_{2}]\ldots[y_{l-1},y_{l}],z_{k-l+3}]$
is a graded identity for $G_{k}$, because
$g_{k-l+2}(z_{1},\ldots,z_{k-l+2})[y_{1},y_{2}]\ldots[y_{l-1},y_{l}]
\in T_{2}(G_{k})$. For modulo $I_{4}$ (identities $(3)$ and $(6)$):
\begin{center}
$[z_{1}\ldots z_{k-l+2}[y_{1},y_{2}]\ldots
[y_{l-1},y_{l}],z_{k-l+3}] + [a[y_{1},y_{2}]\ldots
[y_{l-1},y_{l}],z_{k-l+3}] \equiv 0$,
\end{center}
where $a = \sum_{|T| \mbox{is even and non-empty}}
(-2)^{-\frac{|T|}{2}} f_{T}(z_{1},\ldots,z_{k-l+2})$.

It is well known that $[uv,w] = u[v,w] + [u,w]v$ for $u,v,w \in
F\langle X \rangle$. Therefore, we conclude that:
\begin{center}
$[z_{1}\ldots z_{k-l+2}[y_{1},y_{2}]\ldots
[y_{l-1},y_{l}],z_{k-l+3}] \equiv z_{1}[z_{2}\ldots
z_{k-l+2}[y_{1},y_{2}]\ldots [y_{l-1},y_{l}],z_{k-l+3}] +
z_{2}\ldots z_{k-l+2}[y_{1},y_{2}]\ldots
[y_{l-1},y_{l}][z_{1},z_{k-l+3}] \ \ mod \ \ I_{4}$.
\end{center}
Thus:
\begin{center}
$z_{2}\ldots z_{k-l+2}[z_{1},z_{k-l+3}][y_{1},y_{2}]\ldots
[y_{l-1},y_{l}] \equiv - z_{1}[z_{2}\ldots
z_{k-l+2}[y_{1},y_{2}]\ldots [y_{l-1},y_{l}],z_{k-l+3}] - \newline
[a[y_{1},y_{2}]\ldots [y_{l-1},y_{l}],z_{k-l+3}] \ \ mod \ \ I_{4}$.
\end{center}

By applying successively  the graded identity $[x_{1},x_{2},x_{3}]$
and the expression $[uv,w] = u[v,w] + [u,w]v$, we are done.

\vspace{0.5cm}

(A3). Let $v = z_{1}^{a_{1}}\ldots
z_{n}^{a_{n}}[z_{1},z_{n+1}][y_{1},y_{2}]\ldots[y_{l-1},y_{l}]$ such
that $a_{1} + \ldots + a_{n} = k - l + 1; a_{1},\ldots,a_{n} > 0$.

Choose a convenient graded endomorphism $\phi$ such that
$\phi(z_{1}) = z_{1},\ldots, \newline \phi(z_{k-l+2}) = z_{n},
\phi(z_{k-l+3}) = z_{n+1}$.

For modulo $I_{4}$ (identities (3),(6), and (7)):

\begin{center}
$z_{1}[\phi(z_{1}\ldots z_{k-l+2})[y_{1},y_{2}]\ldots
[y_{l-1},y_{l}],\phi(z_{k-l+3})] + \newline
z_{1}[a[y_{1},y_{2}]\ldots [y_{l-1},y_{l}],\phi(z_{k-l+3})] \equiv
0$,
\end{center}
where $a_{1} = \sum_{|T| \mbox{is even and non-empty}}
(-2)^{-\frac{|T|}{2}} \phi(f_{T}(z_{1},\ldots,z_{k-l+2}))$.

Following word for word the proof of (A2), we conclude that:

\begin{center}
$(a_{1})v \equiv z_{1}(\sum_{J} \beta_{J}
\phi(f_{J}))[y_{1},y_{2}]\ldots[y_{l-1},y_{l}] \ mod \ \ I_{4}$,
\newline for some $\beta_{J} \in F$, $J \subseteq \{1,\ldots,k-l+3\}$.
Moreover if $|J| = 2$, then $Deg_{z_{1}}(\psi(\phi(f_{J}))) = 0$.
\end{center}

Generally,\linebreak if $v \in SS2, deg_{Z}beg(v) + deg_{Y}\psi(v) =
k + 1, 2\mid deg_{Y}\psi(v), \mbox{and} Deg_{pr(z)v} \psi(v) = 1$,
then (by algebraic manipulations):
\begin{center}
$v \equiv \sum_{i=1}^{n} \lambda_{i}v_{i} \ \ mod \ \ I_{4}$,
\end{center}
where $v - \sum_{i=1}^{n} v_{i}$ is a multihomogeneous polynomial
and $v_{1},\ldots,v_{n} \in SS3$.

\vspace{0,5cm}

(A4). The proof is similar to that demonstrated (A2). In this case,
note that due to the graded identities of type (5), we have:
\begin{center}
$[z_{1}\ldots z_{k-l+2},y_{1}]\ldots[y_{l-1},y_{l}] \equiv a.b \ mod
\  I_{4}$, where
\end{center}
\begin{center}
$a = [(\sum_{|T| \mbox{is even and non-empty}}
-(-2)^{-\frac{|T|}{2}} f_{T}(z_{1},\ldots,z_{k-l+2})) ,y_{1}]$;\\
$b = [y_{2},y_{3}]\ldots[y_{l-1},y_{l}]$ (if $l \geq 3$), or $b = 1$
(if $l = 1$).
\end{center}

\vspace{0,5cm}

(A5). This follows from (A4) and identities (5),(6), and (7).

\end{proof}

The next corollary is an immediate consequence of Lemma \ref{gte3}.

\begin{corollary}\label{olavo2}
Let $f = \sum_{i=1}^{n}\lambda_{i}v_{i}$ be a linear combination
from $Pr(X)$. For modulo $I_{4}$, $f$ can be written as:
\begin{center}
$\sum_{i=1}^{m}f_{i}u_{i}$,
\end{center}
where $f_{1},\ldots,f_{m}$ are $p$-polynomials and
$u_{1},\ldots,u_{m} \in SS3$ is (are) distinct.
\end{corollary}

Now, we describe the $\mathbb{Z}_{2}$-graded identities of $G_{k}$.

\begin{theorem}\label{gfinal}
$T_{2}(G_{k}) = I_{4}$.
\begin{enumerate}
\item $[y_{1},y_{2}]\ldots [y_{k},y_{k+1}]$ (if $k$ is odd) \ \ (1);
\item $[y_{1},y_{2}]\ldots [y_{k-1},y_{k}][y_{k+1},x]$ (if $k$
is even and $x \in X - \{y_{1},\ldots,y_{k+1}\})$ \ \ (2);
\item $g_{k-l+2}(z_{1},\ldots,z_{k-l+2})[y_{1},y_{2}]\ldots
[y_{l-1},y_{l}]$ (if $l \leq k$ and $l$ is even) \ \ (3);
\item $g_{k-l+2}(z_{1},\ldots,z_{k-l+2})[z_{k-l+3},y_{1}][y_{2},y_{3}]\ldots[y_{l-1},y_{l}]$
(if $l \leq k$ and $l$ is odd) \ \ (4);
\item
$[g_{k-l+2}(z_{1},\ldots,z_{k-l+2}),y_{1}]\ldots[y_{l-1},y_{l}]$ (if
$l\leq k$ and $l$ is odd) \ \ (5);
\item $[x_{1},x_{2},x_{3}]$ \ \ (6);
\item $z_{1}^{p}$ \ \ (7);
\item $y_{1}^{pq} - y_{1}^{p}$ \ \ (8).
\end{enumerate}
\end{theorem}
\begin{proof}
Let $I_{4}$ be the $T_{2}$-ideal generated by the eight identities
reported above. Suppose by contradiction that $I_{4} \varsubsetneq
T_{2}(G_{K})$. According to Corollary \ref{olavo2}, there exists
polynomial an essential polynomial $f = \sum_{i=1}^{l}f_{i}u_{i} \in
T_{2}(G_{k}) - I_{4}$, where $u_{1},\ldots,u_{n} \in SS3 - \{1\}$.
To avoid repetitive arguments, we suppose that $f_{1} = \ldots =
f_{l} = 1$.

One of the three cases listed below can occur:

\begin{enumerate}
\item $deg_{Z}(beg(LT(f))) + deg_{Y}(\psi(LT(f))) \leq k$;
\item $deg_{Z}(beg(LT(f))) + deg_{Y}(\psi(LT(f))) = k + 1$ and $f$
does not admit a bad term;
\item $deg_{Z}(beg(LT(f))) + deg_{Y}(\psi(LT(f))) = k + 1$ and $f$
admits a bad term.
\end{enumerate}

Case 1. Here, the proof strategy is similar to that of Theorem
\ref{I2}. Consider the following map:


\begin{flushleft}
$\phi: \{y_{1},\ldots,y_{l_{1}},z_{1},\ldots,z_{l_{2}}\} \rightarrow
G \newline$

$y_{1} \mapsto \sum_{l=1}^{a_{1}}e_{k+2l-1}e_{k+2l}
\newline \ldots
\newline y_{n_{1}} \mapsto \sum_{l = a_{1} + \ldots +
a_{n_{1}-1}+1}^{a_{1}+\ldots+a_{n_{1}}} e_{k+2l-1}e_{k+2l} \newline
y_{n_{1} + 1} \mapsto e_{1} + \sum_{l = a_{1} + \ldots +
a_{n_{1}}+1}^{a_{1}+\ldots+a_{n_{1}+1}}e_{k+2l-1}e_{k+2l} \newline
\ldots \newline y_{n_{2}} \mapsto e_{n_{2}-n_{1}} + \sum_{l =
a_{1}+\ldots +
a_{n_{2}-1}+1}^{a_{1}+\ldots+a_{n_{2}}}e_{k+2l-1}e_{k+2l} \newline
y_{n_{2} + 1} \mapsto e_{n_{2} - n_{1} + 1}\newline \ldots \newline
y_{l_{1}} \mapsto e_{l_{1}-n_{1}} \newline z_{1} \mapsto
\sum_{l=1}^{b_{1}}e_{R+l}e_{l_{1}-n_{1}+l} \newline \ldots \newline
z_{m_{1}} \mapsto \sum_{l= b_{1} + \ldots + b_{m_{1}-1}+1}^{b_{1} +
\ldots + b_{m_{1}}}e_{R+l}e_{l_{1}-n_{1}+l}\newline z_{m_{1}+1}
\mapsto e_{R + b_{1} + \ldots + b_{m_{1}} +1} + \sum_{l = b_{1} +
\ldots + b_{m_{1}}+1}^{b_{1}+\ldots
+b_{m_{1}+1}}e_{R+l+1}e_{l_{1}-n_{1}+l} \newline \ldots \newline
z_{m_{2}} \mapsto e_{R + b_{1} + \ldots + b_{m_{2}-1} + m_{2} -
m_{1}} + \sum_{l = b_{1} + \ldots + b_{m_{2}-1}+1}^{b_{1} + \ldots +
b_{m_{2}}}e_{R + m_{2}-m_{1}+l}e_{l_{1}-n_{1}+l}\newline z_{m_{2}+1}
\mapsto e_{S+1}\newline \ldots \newline z_{l_{2}}\mapsto e_{S +
(l_{2} - m_{2})}$,
\end{flushleft}
where $R = k + 2(a_{1} + \ldots + a_{n_{2}})$ and $S = R + (b_{1} +
\ldots + b_{m_{2}}) + m_{2} - m_{1}$.

Here, $dom(\phi(LT(f))) = \lambda(e_{1}\ldots e_{e_{l_{1}-n_{1}+
b_{1} + \ldots + b_{m_{2}}}})(e_{k+1}\ldots e_{S + (l_{2} -
m_{2})})$ for a non-zero $\lambda$. Following word for word the
argument of Theorem \ref{I2}, we can conclude that $dom(\phi(f)) =
\phi(dom(LT(f))) \neq 0$. However, this is a contradiction.



\smallskip

Case 2. Consider the following map (in the map below, we agree to
slight abuse of language: $\sum_{l=1}^{b_{1}-1}e_{k+l+1}e_{l} = 0$,
when $b_{1} = 1$):

\begin{flushleft}
$\phi: \{y_{1},\ldots,y_{l_{1}},z_{1},\ldots,z_{l_{2}}\} \rightarrow
G \newline$ $z_{1} \mapsto e_{k+1} +
\sum_{l=1}^{b_{1}-1}e_{k+l+1}e_{l}
\newline \ldots \newline z_{m_{1}} \mapsto \sum_{l = b_{1} + \ldots
+ b_{m_{1}-1}}^{b_{1}+\ldots+b_{m_{1}}-1}e_{k + l + 1}e_{l} \newline
z_{m_{1}+1} \mapsto e_{k + b_{1} +\ldots + b_{m_{1}}+1} +
\sum_{l=1}^{b_{m_{1}+1}}e_{k + b_{1} + \ldots + b_{m_{1}}+l+1}e_{l +
b_{1}+\ldots+b_{m_{1}}-1} \newline \ldots \newline z_{m_{2}} \mapsto
e_{k + b_{1} + \ldots + b_{m_{2}-1} + (m_{2}-m_{1})} + \sum_{l =
1}^{b_{m_{2}}}e_{k + b_{1} + \ldots + b_{m_{2}-1} + (m_{2}-m_{1}) +
l}e_{l + b_{1} + \ldots + b_{m_{2}-1} -1} \newline z_{m_{2}+1}
\mapsto e_{k + b_{1} +\ldots + b_{m_{2}}+ m_{2}-m_{1} +1},
\ldots,z_{l_{2}} \mapsto e_{k + b_{1} +\ldots + b_{m_{2}} + l_{2} -
m_{1}} \newline y_{1} \mapsto \sum_{l=1}^{a_{1}}e_{M + 2l -
1}e_{M+2l} \newline y_{n_{1}} \mapsto \sum_{l = a_{1} + \ldots +
a_{n_{1}-1} + 1}^{a_{1} + \ldots + a_{n_{1}}}e_{M + 2l - 1}e_{M +
2l} \newline y_{n_{1} + 1} \mapsto e_{b_{1} + \ldots + b_{m_{2}}} +
\sum_{l= a_{1} + \ldots + a_{n_{1}} + 1}^{a_{1} + \ldots +
a_{n_{1}+1}}e_{M + 2l - 1}e_{M+2l}\newline y_{n_{2}} \mapsto
e_{b_{1} + \ldots + b_{m_{2}} + (n_{2}-n_{1}-1)} + \sum_{l = a_{1} +
\ldots + a_{n_{2}-1}+1}^{a_{1} + \ldots + a_{n_{2}}}e_{M + 2l -
1}e_{M + 2l} \newline y_{n_{2} + 1} \mapsto e_{b_{1}+ \ldots +
b_{m_{2}} + n_{2}-n_{1}} \newline \ldots
\newline y_{l_{1}} \mapsto e_{b_{1}+ \ldots + b_{m_{2}} +
l_{1}-n_{1}-1}$,
\end{flushleft}
where $M = k + b_{1} +\ldots + b_{m_{2}} + l_{2} - m_{1}$. Notice
that
\begin{center}
$dom(\phi(LT(f))) = \lambda (e_{1}\ldots e_{b_{1}+ \ldots +
b_{m_{2}} + l_{1}-n_{1}-1})(e_{k+1}\ldots e_{M + 2(a_{1} + \ldots +
a_{n_{2}})})$
\end{center}
for some non-zero $\lambda$. Suppose that there exists $u_{i} \neq
LT(f)$. If there exists $x \in Variable(f)$ such that $Deg_{x}u_{i}
< Deg_{x} LT(f)$, it is easy to see that no summand of $\phi(u_{i})$
contains $supp(\phi(x))$. Otherwise, $u_{i} - LT(f)$ is
multihomogeneous and there exists $x \in Variable(f)$ such that
$Deg_{x} beg(u_{i}) < Deg_{x} beg (LT(f))$. Bearing in mind that
$u_{i}$ is not a bad term, we can suppose that $x \neq
pr(z)(LT(f))$. Therefore, we have $\phi(u_{i}) = 0$. Hence,
$dom(\phi(f)) = dom(\phi(LT(f))) \neq 0$. This is a contradiction.


\smallskip

Case 3. In this situation, notice that
\begin{center}
$deg_{Z}(beg(LBT(f))) + deg_{Y}(\psi(LBT(f))) \leq k$.
\end{center}
Consider $\phi : \{y_{1},\ldots,y_{l_{1}},z_{1},\ldots,z_{l_{2}}\}
\rightarrow G$ as in the Case 1. We have the following:
\begin{center}
$dom(\phi(LBT(f))) = \lambda(e_{1}\ldots e_{e_{l_{1}-n_{1}+ b_{1} +
\ldots + b_{m_{2}}}})(e_{k+1}\ldots e_{S + (l_{2} - m_{2})})$
\end{center}
for a non-zero $\lambda$. It is clear that if $u_{i} < LBT(f)$, no
summand of $\phi(u_{i})$ contains $supp(dom(\phi(LBT(f))))$. On the
other hand, if $u_{i} > LBT(f)$, $u_{i}$ is not a bad term.
Furthermore, $u_{i} = LT(f)$ or $LBT(F) < u_{i} < LT(f)$. In the
first case, there exists a variable $y \in Y\cap Variable(f)$ such
that $Deg_{y}(beg(LBT(f))) > Deg_{y} (beg(LT(f)))$. Therefore,
$\phi(LT(f)) = 0$. In the second case, $deg(LT(f)) = deg(LBT(f)) =
deg(u_{i})$ and $Deg_{x}beg(LBT(f)) = Deg_{x} beg(u_{i}) =
Deg_{x}beg(LT(f))$ for all $x \in Z - \{pr(z)(LT(f))\}$. If $u_{i} -
LBT(f)$ is not multihomogeneous, there exists a variable $x \in
Variable(f)$ such that $Deg_{x} LBT(f) > Deg_{x} u_{i}$. Thus, no
summand of $\phi(u_{i})$ contains $supp(\phi(x))$. Now, assume
$u_{i} - LBT(f)$ is multihomogeneous. Notice that
$Deg_{pr(z)(LT(f))} beg(LT(f)) = Deg_{pr(z)(LT(f))} beg(u_{i})$ or
$Deg_{pr(z)(LT(f))} beg(LBT(f)) = Deg_{pr(z)(LT(f))} beg(u_{i})$.
Therefore, there exists $y \in Y\cap Variable(f)$ such that
$Deg_{y}(beg(LBT(f)))
> Deg_{y}(beg(u_{i}))$. Hence, $\phi(u_{i}) = 0$.






From these three cases, we have $I_{4} = T_{2}(E)$, as required.
\end{proof}


\section{Acknowledgments}

This research was supported by CNPq, Conselho Nacional de
Desenvolvimento Científico e Tecnológico, Brazil.


\end{document}